\pdfoutput=1
\RequirePackage{ifpdf}
\ifpdf 
\documentclass[pdftex]{sigma}
\else
\documentclass{sigma}
\fi

\numberwithin{equation}{section}

\newtheorem{Theorem}{Theorem}[section]
\newtheorem{Corollary}[Theorem]{Corollary}
\newtheorem{Lemma}[Theorem]{Lemma}
\newtheorem{Proposition}[Theorem]{Proposition}

\begin{document}

\allowdisplaybreaks

\newcommand{\arXivNumber}{1909.13262}

\renewcommand{\thefootnote}{}

\renewcommand{\PaperNumber}{091}

\FirstPageHeading

\ShortArticleName{Locally Nilpotent Derivations of Free Algebra of Rank Two}

\ArticleName{Locally Nilpotent Derivations\\
of Free Algebra of Rank Two\footnote{This paper is a~contribution to the Special Issue on Algebra, Topology, and Dynamics in Interaction in honor of Dmitry Fuchs. The full collection is available at \href{https://www.emis.de/journals/SIGMA/Fuchs.html}{https://www.emis.de/journals/SIGMA/Fuchs.html}}}

\Author{Vesselin DRENSKY~$^\dag$ and Leonid MAKAR-LIMANOV~$^{\ddag\S}$}

\AuthorNameForHeading{V.~Drensky and L.~Makar-Limanov}

\Address{$^\dag$~Institute of Mathematics and Informatics, Bulgarian Academy of Sciences, 1113 Sofia, Bulgaria}
\EmailD{\href{mailto:drensky@math.bas.bg}{drensky@math.bas.bg}}

\Address{$^\ddag$~Department of Mathematics, Wayne State University Detroit, MI 48202, USA}
\EmailD{\href{mailto:lml@math.wayne.edu}{lml@math.wayne.edu}}

\Address{$^\S$~Department of Mathematics, The Weizmann Institute of Science, Rehovot 7610001, Israel}

\ArticleDates{Received October 01, 2019, in final form November 13, 2019; Published online November 18, 2019}

\Abstract{In commutative algebra, if $\delta$ is a locally nilpotent derivation of the polynomial algebra $K[x_1,\ldots,x_d]$ over a field $K$ of characteristic 0 and $w$ is a nonzero element of the kernel of $\delta$, then $\Delta=w\delta$ is also a locally nilpotent derivation with the same kernel as $\delta$. In this paper we prove that the locally nilpotent derivation $\Delta$ of the free associative algebra $K\langle X,Y\rangle$ is determined up to a multiplicative constant by its kernel. We show also that the kernel of $\Delta$ is a free associative algebra and give an explicit set of its free generators.}

\Keywords{free associative algebras; locally nilpotent derivations; algebras of constants}

\Classification{16S10; 16W25; 16W20; 13N15}

\rightline{\it To the 80th anniversary of Dmitry Fuchs}

\renewcommand{\thefootnote}{\arabic{footnote}}
\setcounter{footnote}{0}

\section{Introduction}
Let $K$ be a field of characteristic 0. Locally nilpotent derivations $\delta$ of polynomial algebras $K[x_1,\ldots,x_d]$ and their kernels $\ker(\delta)$ are subjects of active investigation. Traditionally, the kernel of a derivation $\delta$ of $K[x_1,\ldots,x_d]$ is called the algebra of constants of $\delta$ and is denoted by $K[x_1,\ldots,x_d]^{\delta}$. The algebras of constants of locally nilpotent derivations play an essential role in the study of the automorphism group of $K[x_1,\ldots,x_d]$,
including the generation of $\operatorname{Aut}(K[x,y])$ by tame automorphisms, the Jacobian conjecture, in invariant theory, fourteenth Hilbert's problem and other important topics. See the books by Nowicki~\cite{No1}, van den Essen~\cite{E}, and Freudenburg~\cite{Fr2} for details. In particular, using locally nilpotent derivations,
Rentschler~\cite{R} gave an easy proof of the theorem of Jung--van der Kulk~\cite{J, K} that all automorphisms of $K[x,y]$ are tame. Another natural proof based on locally nilpotent derivations was given by Makar-Limanov~\cite{ML2}, see also the book~\cite{D1}. The most natural way to define the Nagata automorphism~\cite{Na}
\[
(x,y,z)\to\big(x-2\big(xz+y^2\big)y-\big(xz+y^2\big)^2z,y+\big(xz+y^2\big)z,z\big)
\]
is also in terms of locally nilpotent derivations, see Bass~\cite{B} and Smith~\cite{Sm}. The famous Jacobian conjecture is equivalent to several conjectures stated in the language of locally nilpotent derivations, see~\cite{E}. Several nice counterexamples to fourteenth Hilbert's problem are obtained as algebras of constants of locally nilpotent derivations, see the survey and the book by Freudenburg~\cite{Fr1, Fr2} and the survey by Nowicki~\cite{No2}. On the other hand, the well known theorem of Weitzenb\"ock~\cite{W} states that if $\delta$ is a nilpotent linear operator acting on the $d$-dimensional vector space $Kx_1\oplus \cdots\oplus Kx_d$, then the algebra of constants of the locally nilpotent derivation of $K[x_1,\ldots,x_d]$ which extends $\delta$ is a finitely generated algebra. A modern proof of the theorem is given by Seshadri~\cite{Ses}, with further simplification by Tyc~\cite{T}, see also~\cite{No1}. Clearly, the algebra of constants $K[x_1,\ldots,x_d]^{\delta}$ coincides with the algebra of invariants of the linear operator
\[
\exp(\delta)=1+\frac{\delta}{1!}+\frac{\delta^2}{2!}+\cdots .
\]

If $\delta$ is a locally nilpotent derivation of $K[x_1,\ldots,x_d]$ and $0\not=w\in K[x_1,\ldots,x_d]^{\delta}$, then $\Delta=w\delta$ is also a locally nilpotent derivation with the same algebra of constants as~$\delta$. In particular, starting from the Weitzenb\"ock derivation of $K[x,y,z]$ defined by
\[
\delta(x)=-2y,\qquad \delta(y)=z,\qquad \delta(z)=0,
\]
$w=xz+y^2\in K[x,y,z]^{\delta}$, and $\Delta=\big(xz+y^2\big)\delta$ one obtains the Nagata automorphism as $\exp(\Delta)$. We would like to mention that Shestakov and Umirbaev~\cite{SU1, SU2} proved the Nagata conjecture that the Nagata automorphism is wild with methods of noncommutative algebra.

Locally nilpotent derivations of free associative algebras $K\langle X_1,\ldots,X_d\rangle$ have not been studied as intensively as in the commutative case. We shall mention the old result of Falk \cite{Fa} who described the intersection of the kernels of the formal partial derivatives $\partial/\partial X_j$ of $K\langle X_1,\ldots,X_d\rangle$, and the relations of the formal partial derivatives with theory of algebras with polynomial identity due to Specht \cite{Sp}, see also \cite{D1} for further development. Drensky and Gupta~\cite{DG} studied the kernels of Weitzenb\"ock derivations of $K\langle X_1,\ldots,X_d\rangle$ and established that in all nontrivial cases the kernel is not finitely generated. As in the case of polynomial algebras, the candidate for a~wild automorphism, the automorphism of Anick \cite[p.~343]{Co}
\[
(X,Y,Z)\to(X+Z(XZ-ZY),Y+(XZ-ZY)Z,Z)
\]
can also be expressed as $\exp(\Delta)$ for the locally nilpotent derivation $\Delta$ of $K\langle X,Y,Z\rangle$ defined by
\begin{gather*}
\Delta(X)=Z(XZ-ZY),\qquad\Delta(Y)=(XZ-ZY)Z,\qquad\Delta(Z)=0.
\end{gather*}
The wildness of the Anick automorphism was established by Umirbaev \cite{U}.

In this paper we study locally nilpotent derivations $\Delta$ of the free unitary associative algebra $K\langle X,Y\rangle$ over a field $K$ of characteristic~0. As in the commutative case we shall call the kernel of $\Delta$ the algebra of constants of~$\Delta$ and denote it by $K\langle X,Y\rangle^{\Delta}$. Our main result is that the locally nilpotent derivations of $K\langle X,Y\rangle$ are determined up to a multiplicative constant by their algebras of constants.

It is easy to see that $\Delta$ is of the form $\Delta(U)=0$, $\Delta(V)=f(U)$, with respect to a suitable system of generators $U$, $V$ of $K\langle X,Y\rangle$. This follows from the description of Rentschler \cite{R} of the locally nilpotent derivations of $K[x,y]$ and the isomorphism of the automorphism groups of $K[x,y]$ and $K\langle X,Y\rangle$ which is a consequence of the theorem of Jung--van der Kulk~\cite{J, K} and its analogue for the automorphisms of $K\langle X,Y\rangle$ due to Czerniakiewicz \cite{Cz1, Cz2} and Makar-Limanov~\cite{ML}. This result is similar to the recent description of locally nilpotent derivations of the free Poisson algebra with two generators given by Makar-Limanov, Turusbekova, and Umirbaev~\cite{MLTU}.

As a consequence of the result of Lane~\cite{L} and Kharchenko~\cite{Kh} the algebra of constants $K\langle X,Y\rangle^{\Delta}$ of the nontrivial Weitzenb\"ock derivation $\Delta$ of $K\langle X,Y\rangle$ is a free associative algebra. A set of free generators of this algebra was given by Drensky and Gupta~\cite{DG}. We generalize this result and show that the algebra $K\langle X,Y\rangle^{\Delta}$ is free for any locally nilpotent derivation $\Delta$ of $K\langle X,Y\rangle$. As in \cite{DG} we give an explicit set of free generators of $K\langle X,Y\rangle^{\Delta}$. See also \cite{Jo} where it is shown that $K\langle X,Y\rangle^{\Delta}$ is a free associative algebra for a nontrivial homogeneous derivation (and from which the freeness in our case can be deduced).

\section{Preliminaries}\label{section2}

For an algebra $R$ over a field $K$ a linear operator $\delta\colon R\rightarrow R$ is called a derivation if it satisfies the Leibniz law $\delta(ab) = \delta(a)b + a\delta(b)$.
The kernel of a derivation $\delta$ is denoted by~$R^{\delta}$ and the elements of the kernel are called $\delta$-constants (or just constants when this is not confusing).
A~derivation~$\delta$ is called locally nilpotent if for any $r \in R$ there exists a natural number~$n$ (which depends on~$r$) for which $\delta^n(r) = 0$. The function
\[
\deg(r) = \max\big(d\,|\, \delta^d(r) \neq 0\big), \qquad \deg(0) = -\infty,
\]
is a degree function with familiar properties:
\begin{gather*}
\deg(r_1r_2) = \deg(r_1) + \deg(r_2),\qquad \deg(r_1 + r_2) \leq \max(\deg(r_1), \deg(r_2)),\\
\deg(r_1 + r_2) =\max(\deg(r_1), \deg(r_2)) \qquad \text{when} \quad \deg(r_1) \neq \deg(r_2),\\
\deg(\delta(r)) = \deg(r) - 1\qquad \text{if} \quad \delta(r) \neq 0.
\end{gather*}

The set of all lnds (locally nilpotent derivations) of $R$ is denoted by $\text{LND}(R)$.

The intersection $\bigcap R^{\delta}$, $\delta \in \text{LND}(R)$, of kernels of all locally nilpotent derivations of $R$ is denoted by $\text{AK}(R)$ (absolute Konstanten of~$R$, sometimes denoted as $\text{ML}(R)$).

If $\delta \in \text{LND}(R)$ and characteristic of $K$ is zero then the linear operator
\[
\exp(\delta)=1+\frac{\delta}{1!}+\frac{\delta^2}{2!}+\cdots
\]
is an automorphism of $R$.

In the sequel we fix a field $K$ of characteristic 0 and consider the polynomial algebra $K[x,y]$ and the free associative algebra $K\langle X,Y\rangle$. Let
\[
\pi\colon \ K\langle X,Y\rangle\to K[x,y]
\]
be the natural homomorphism. We denote the elements $U$, $V$, etc.\ of $K\langle X,Y\rangle$ by upper case symbols and their images under $\pi$ by the same lower case symbols $u$, $v$, etc. Let $C$ be the commutator ideal of $K\langle X,Y\rangle$. It is generated by the commutator
\[
T_1 = [Y, X] = YX - XY.
\]

By the theorem of Jung--van der Kulk \cite{J, K}, the automorphisms of $K[x,y]$ are tame, i.e., are compositions of affine automorphisms
\[
x \rightarrow a_1 x + a_2 y + a_3, \qquad y \rightarrow b_1 x + b_2 y + b_3, \qquad a_i,b_i\in K,\qquad a_1b_2 - a_2b_1 \neq 0,
\]
and triangular automorphisms
\[
x \rightarrow x, \qquad y \rightarrow y + p(x),\qquad p(x)\in K[x].
\]

A similar theorem of Czerniakiewicz \cite{Cz1, Cz2} and Makar-Limanov \cite{ML} states that the automorphisms of $K\langle X,Y\rangle$ are also tame. Therefore
\[
\Psi(T_1) = cT_1, \qquad c \in K^{\ast},
\]
for any automorphism $\Psi$ of $K\langle X,Y\rangle$ (indeed, just check that this is true for affine and triangular automorphisms).

The structure of the automorphism groups of $K[x, y]$ and $K\langle X,Y\rangle$ is also known, it is a~free product of the subgroups of affine and triangular automorphisms with amalgamation along their intersection~\cite{Se}. So we can think that there is a group $H$ isomorphic to $\operatorname{Aut}K[x,y]$ and $\operatorname{Aut}K\langle X,Y\rangle$ which acts on $K[x, y]$ and $K\langle X,Y\rangle$.

Any automorphism of $K\langle X,Y\rangle$ induces an automorphism of $K[x, y]$ and, since the structure of the group $H$ insures that this is one to one correspondence, any automorphism of $K[x,y]$ can be uniquely lifted to an automorphism of $K\langle X,Y\rangle$.

We shall use below a lexicographic ordering of monomials of $K\langle X,Y\rangle$ defined by $Y\gg X > 1$ and denote by $\overline{S}$ the leading monomial of $S \in K\langle X,Y\rangle$.

In the sequel we shall show that we can reduce our considerations to the case when the lnd~$\Delta$ is such that
\[
\Delta(X)=0,\qquad \Delta(Y)=F=f(X),
\]
where $0\not=f(x)\in K[x]$. In this special case we shall define the operator $\boxdot$ on $K\langle X,Y\rangle$ by
\[
\boxdot(A)= YAF - FAY,\qquad A\in K\langle X,Y\rangle,
\]
and shall fix the sequence $T_1,T_2,\ldots$, starting with $T_1= YX - XY$ and then inductively
\[
T_{i+1} = \boxdot^i(T_1).
\]

\section{Description of locally nilpotent derivations}

Though the lnds of $K\langle X,Y\rangle$ are similar to the lnds of $K[x,y]$ there are also significant differences.

It is quite clear that $\text{AK}(K[x, y]) = K$ (just observe that the partial derivatives $\frac{\partial}{\partial x}$ and $\frac{\partial}{\partial y}$ are locally nilpotent)
but we shall show later that $\text{AK}(K\langle X,Y\rangle) = K[T_1]$. The following lemma shows that $\text{AK}(K\langle X,Y\rangle) \supseteq K[T_1]$.

\begin{Lemma}\label{commutator goes to 0} $\delta(T_1) = 0$ for any lnd of $K\langle X,Y\rangle$.
\end{Lemma}

\begin{proof} If $\delta \in \text{LND}(K\langle X,Y\rangle)$ then $\lambda\delta \in \text{LND}(K\langle X,Y\rangle)$ for any $\lambda \in K$. Take $\Psi_{\lambda} = \exp(\lambda\delta)$; then $\Psi_{\lambda}([Y,X]) = c(\lambda)[Y,X]$, where $c(t) \in K[t]$ (recall that~$\delta$ is an lnd). On the other hand $\Psi_{\lambda} \Psi_{\mu} = \Psi_{\lambda + \mu}$, i.e., $c(s)c(t) =c(s + t)$. Since $c(s) \ne 0$ this is possible only if $c(t) = 1$. Hence $\delta([Y,X]) = 0$.
\end{proof}

Now we shall prove that lnds of $K\langle X,Y\rangle$ are similar to those of $K[x,y]$.

\begin{Proposition}\label{similarity} Let $\Delta$ be a locally nilpotent derivation of $K\langle X,Y\rangle$. Then there is a system of generators $U$, $V$ of $K\langle X,Y\rangle$ and a polynomial $f(U)$ depending on~$U$ only, such that $\Delta(U)=0$, $\Delta(V)=f(U)$.
\end{Proposition}

\begin{proof} Let $\Delta$ be a locally nilpotent derivation of $K\langle X,Y\rangle$. Clearly, $\Delta$ induces a locally nilpotent derivation $\delta$ of $K[x,y]$.
By the theorem of Rentschler~\cite{R}, $K[x,y]$ has a system of generators~$u$,~$v$ such that $\delta(u)=0$, $\delta(v)=f(u)$ for some $f(u)\in K[u]$.

As was mentioned above this pair of generators can be uniquely lifted to the pair $U$, $V$ of generators of $K\langle X,Y\rangle$.

Let us consider the automorphisms
\[
\varPhi=\exp(\Delta)\in \operatorname{Aut}K\langle X,Y\rangle=\operatorname{Aut}K\langle U,V\rangle
\]
and
\[
\varphi=\exp(\delta) =1+\frac{\delta}{1!}+\frac{\delta^2}{2!}+\cdots\in\operatorname{Aut}K[x,y]=\operatorname{Aut}K[u,v].
\]
Then
\[
\varphi\colon \ u\to u,\qquad \varphi\colon \ v\to v+f(u).
\]

From the uniqueness mentioned in Section~\ref{section2}
\[
\varphi(u) = u,\qquad \varphi(v) = v+f(u)
\]
implies $\varPhi(U) = U$, $\varPhi(V) = V+f(U)$. Since $\varPhi=\exp(\Delta)=1+\varTheta$, where
\[
\varTheta=\frac{\Delta}{1!}+\frac{\Delta^2}{2!}+\cdots
\]
and $\varTheta^n(S)=0$ for $S\in K\langle X,Y\rangle$ and a sufficiently large $n$, we have that
\[
\Delta=\log(1+\varTheta)=\frac{\varTheta}{1}-\frac{\varTheta^2}{2}+\frac{\varTheta^3}{3}-\cdots
\]
and $\varPhi$ determines uniquely the lnd $\Delta$. Hence $\Delta(U) = 0$, $\Delta(V) = f(U)$.
\end{proof}

Another difference between the locally nilpotent derivations of $K[x,y]$ and $K\langle X,Y\rangle$ is that in the latter case they can be distinguished by their algebras of constants.

\begin{Theorem}\label{derivations with the same kernel}
Let $\Delta_1$ and $\Delta_2$ be two non-zero locally nilpotent derivations of $K\langle {X,Y}\rangle$.\
Then~$\Delta_1$ and~$\Delta_2$ have the same algebras of constants if and only if $\Delta_2=\alpha\Delta_1$ for a nonzero $\alpha\in K$.
\end{Theorem}

\begin{proof} Changing the generators of $K\langle X,Y\rangle$, by Proposition~\ref{similarity} we may assume that \mbox{$\Delta_1(X){=}0$}, $\Delta_1(Y)=f(X)=F$ for some nonzero $F=f(X)\in K\langle X,Y\rangle$. Since $K\langle X,Y\rangle^{\Delta_1}=K\langle X,Y\rangle^{\Delta_2}$ we have that $\Delta_2(X) = 0$. By Lemma~\ref{commutator goes to 0}
\[
\Delta_2(T_1) = [\Delta_2(Y),X]+[Y,\Delta_2(X)] = [\Delta_2(Y),X] = 0.
\]
Therefore $\Delta_2(Y)=g(X)= G$. A direct computation gives that
\[
T_2 = YT_1F - FT_1Y \in K\langle X,Y\rangle^{\Delta_1}.
\]
Hence $\Delta_2(T_2)=GT_1F - FT_1G = g(X)T_1f(X)-f(X)T_1g(X) = 0$ which implies that $g(x)=\alpha f(x)$ for some $\alpha\in K$. Therefore $\Delta_2=\alpha \Delta_1$. Since $\Delta_1,\Delta_2\not=0$, we obtain that $\alpha\not=0$.
\end{proof}

\section[Algebras of constants of derivations of $K\langle X,Y\rangle$]{Algebras of constants of derivations of $\boldsymbol{K\langle X,Y\rangle}$}

By Proposition \ref{similarity}, up to a change of the free generators of $K\langle X,Y\rangle$ every nontrivial locally nilpotent derivation $\Delta$ of $K\langle X,Y\rangle$ is of the form
\[
\Delta(X)=0,\qquad \Delta(Y)=f(X),
\]
where $0\not=f(x)\in K[x]$. In the sequel we shall fix $\deg(f)=m\geq 0$ and $\Delta$ as defined above.

\begin{Proposition}\label{AK is generated by commutator}
$\text{AK}(K\langle X,Y\rangle) = K[T_1]$.
\end{Proposition}
\begin{proof}
Let us consider derivations
\[
\delta_m\colon \ \delta_m(X) = 0, \qquad\delta_m(Y) = X^m.
\]
Suppose $\delta_m(P) = 0$ for all $m$. We may assume that $P$ is homogeneous relative to~$X$ and $Y$. Write $P = XP_0 + YP_1$, then
\[
0 = \delta_m(P) = X\delta_m(P_0) + X^mP_1+ Y\delta_m(P_1).
\]
Hence $\delta_m(P_1) = 0$ and we can assume by induction on $\deg_Y$ that $P_1$ belongs to the subalgebra $K\langle X, T_1\rangle$ of $K\langle X,Y\rangle$ generated by $X$ and $T_1$ and write $P_1 = XP_{10} + T_1P_{11}$. If $P_{11} \neq 0$ then $\overline{X^mT_1P_{11}}$ cannot be canceled by any monomial of $X\delta_m(P_0)$ if $m$ is sufficiently large. Hence $P_{11} = 0$ and $P_{10} \in K\langle X, T_1\rangle$. Therefore
\[
P = XP_0 + YXP_{10} =XP_0 + T_1P_{10} + XYP_{10} = X(P_0 + YP_{10}) + T_1P_{10}.
\]
Then $\delta_m(P_0 +YP_{10}) = 0$ because $T_1P_{10} \in K\langle X,T_1\rangle$ and we can assume by induction on $\deg_X$ that $P_0 + YP_{10} \in K\langle X, T_1\rangle$, i.e., $P \in K\langle X, T_1\rangle$. Of course
\[
\text{AK}(K\langle X,Y\rangle) \subseteq K\langle X, T_1\rangle\bigcap K\langle Y, T_1\rangle= K[T_1]
\]
since we can switch $X$ and $Y$.
\end{proof}

Consider the operator $\boxdot$ on $K\langle X,Y\rangle$ defined in Section~\ref{section2}. We shall prove in this section that the algebra of constants of $\Delta$ is the minimal algebra $R_F$ which contains $K\langle X,T_1\rangle$ and is closed under this operator. Since $\boxdot\Delta = \Delta\boxdot$ it is clear that $R_F \subseteq K\langle X,Y\rangle^{\Delta}$. It is worth observing that the kernel of~$\boxdot$ is $K[Y]$ if $\deg_X(F) = 0$ and $0$ if $\deg_X(F) > 0$ and that $\deg(\boxdot(A)) = \deg(A)$
(where $\deg$ is the degree function induced by $\Delta$) if $\deg_X(F) > 0$. We shall also denote~$\boxdot(A)$ by~$\{A\}$. This bracketing is a bit unusual since
$\boxdot^n(A)$ will be recorded as $\{\{\ldots\{A\}\ldots\}\}$ with the same number $n$ of the left and right brackets and there can be more than two terms inside of a~pair of brackets, but as in the ordinary bracketing in a~configuration of three brackets like this $\{A_1\{A_2\}$ the first bracket cannot match the third bracket, it should be matched by a bracket $\}$ to the right of the third bracket and second and third brackets are matched.

\begin{Theorem}\label{preparing of freedom}Let $L\in K\langle X,Y\rangle$. If $\Delta^n(L) = 0$ then $L$ belongs to the linear span $R_F^n$ of elements $A_1YA_2Y\cdots YA_k$, where $k \leq n$ and each $A_i$, $1 \leq i\leq k$, is a monomial from $R_F$, endowed with an arbitrary number of matching pairs of brackets $\{\}$.
\end{Theorem}

\begin{proof} We consider two cases separately.

(a) $m = 0$ (we can assume that $\Delta(Y) = 1$). Consider the sequence of elements $T_1, \ldots, T_i, \ldots$ defined in Section~\ref{section2} by $T_1= YX - XY$, $T_{i+1} = \boxdot^i(T_1)$. In this case $\overline{T_i} = Y^iX$ and any element $S \in K\langle X,Y\rangle$ can be written as $S = \sum\limits_{j=0}^k S_jY^j$ where $S_j \in K\langle X,T_1, \ldots, T_i, \ldots\rangle$. Since $\Delta(S) = \sum\limits_{j=0}^k jS_j Y^{j-1}$, $\Delta^n(S) = 0$, and $\Delta^k(S) \neq 0$ if $S_k \neq 0$ it is clear that $k < n$.

(b) $m > 0$. Let us introduce a weight degree function on $K\langle X,Y\rangle$ by $w(X) =1$, $w(Y) = m$. Then the space $V_N$ spanned by monomials of the weight not exceeding $N$ is mapped by the derivation into itself. We proceed by induction on $w(S)$. If $w(S)$ is sufficiently small, say does not exceed $m$, the claim is obvious. Assume that for the weight less than $N$ the claim is true.

Take an $L$ for which $w(L) = N$ and $L^{(k)} = 0$ (here and further on $L^{(k)}$ denotes $\Delta^{k}(L)$). We can assume that $L(X, 0) = 0$ and write
\[
L =L_mF + \sum_{i=0}^{m-1}L_iYX^i.
\]
Then
\[
L^{(k)}_mF + k\sum_{i=0}^{m-1} L_i^{(k-1)}X^iF + \sum_{i=0}^{m-1}L_i^{(k)}YX^i = 0.
\]
Hence $L_i^{(k)} = 0$ for $i < m$ and
\[
\left(L_m' + k\sum_{i=0}^{m-1} L_iX^i\right)^{(k-1)} = 0.
\]
Therefore $\widehat{L}^{(k)} = 0$ for $\widehat{L} = L_mF + \sum\limits_{i=0}^{m-1}L_iX^iY$.

It is sufficient to check the claim for $\widehat{L}$ since $L-\widehat{L} =\sum\limits_{i=0}^{m-1}L_i[Y,X^i]$ satisfies the claim by induction ($w(L_i) < N$ and $[Y, X^i] \in R_F$).

Write $\widehat{L} = L_mF + H_0Y$. Then $H_0^{(k)} = 0$ and $(L_m' + kH_0)^{(k-1)} =0$. Hence $L_m^{(k+1)} = 0$ and $\widetilde{L}^{(k)} = 0$ for $\widetilde{L} =kL_mF - L_m'Y$. It is sufficient to check the claim for $\widetilde{L}$ since $k\widehat{L} - \widetilde{L} = (kH_0 + L_m')Y$ and $kH_0 + L_m'$ satisfy the claim by induction.

Since $L_m^{(k+1)} = 0$ and $w(L_m) < N$ we can write
\[
L_m = \sum_{\textbf{j}}\alpha_{j_0}Y\alpha_{j_1}Y \cdots Y\alpha_{j_k} + S,
\]
where $\alpha_{j_i} \in R_F$, the summands are endowed with brackets $\{\}$, and $S$ is the sum of terms in which $Y$ appears less than $k$ times. We can omit $S$ since $kSF - S'Y \in R^k_F$.

Take one of the summands $\mu_{\textbf{j}}$ and consider $\nu_{\textbf{j}} =k\mu_{\textbf{j}} F - \mu_{\textbf{j}}'Y$. Since $\Delta$ and $\boxdot$ commute
\[
\nu_{\textbf{j}} = k\mu_{\textbf{j}} F - \sum_{i=1}^k \alpha_{j_0}Y\alpha_{j_1}Y \cdots \alpha_{j_{i-1}}F\alpha_{j_i} Y \cdots Y\alpha_{j_k}Y,
\]
where each term $\alpha_{j_0}Y\alpha_{j_1}Y \cdots \alpha_{j_{i-1}}F\alpha_{j_i} Y \cdots Y\alpha_{j_k}Y$ has the same bracketing as $\mu=\mu_{\textbf{j}}$.

Consider $P_i = \mu F - \alpha_{j_0}Y\alpha_{j_1}Y \cdots \alpha_{j_{i-1}}F\alpha_{j_i}Y \cdots Y\alpha_{j_k}Y$. It is clear that $P_i^{(k)} = 0$ so we should check that $P_i$ can be recorded as a sum of terms containing only $k-1$ entries of $Y$ (we do not count $Y$'s appearing in $\boxdot$).

Write $\mu = V_1YU_1$ where $Y$ is the one which is replaced by $F$ in $P_i$ and introduce two operations:
\[
\bigtriangledown_{r,U}(V_1YU_1) = V_1YU_1UF - V_1FU_1UY\qquad \text{and} \qquad \bigtriangledown_{l,U}(V_1YU_1) = FUV_1YU_1 - YUV_1FU_1.
\]
We shall write $\bigtriangledown_r$ and $\bigtriangledown_l$ when $U = 1$, so $P_i = \bigtriangledown_r(V_1YU_1)$.

The operator $\boxdot$ is defined on all algebra while the operations $\bigtriangledown_{r,U}$, $\bigtriangledown_{l,U}$ are defined only on specially recorded elements and their extension does not seem to be canonical.

Assume that $V_1YU_1 = \boxdot(V_2YU_2)$. Then we need to simplify $\bigtriangledown_r(\boxdot(V_2YU_2))$. In order to do this let us compute $[\bigtriangledown_r,\boxdot](V_2YU_2)$.

This is a bit tedious but not difficult:
\begin{gather*}
\bigtriangledown_r(\boxdot(V_2YU_2)) = [Y(V_2YU_2)F - F(V_2YU_2)Y]F -[Y(V_2FU_2)F - F(V_2FU_2)Y]Y,\\
\boxdot(\bigtriangledown_r(V_2YU_2)) = Y[(V_2YU_2)F - (V_2FU_2)Y]F -F[(V_2YU_2)F - (V_2FU_2)Y]Y.
\end{gather*}
Hence
\begin{gather*}
[\bigtriangledown_r, \boxdot](V_2YU_2) = - F(V_2YU_2)YF +F(V_2YU_2)FY - Y(V_2FU_2)FY + Y(V_2FU_2)YF\\
\hphantom{[\bigtriangledown_r, \boxdot](V_2YU_2) =}{}
=[Y(V_2FU_2) - F(V_2YU_2)][Y,F] = -\bigtriangledown_l(V_2YU_2)[Y,F].
\end{gather*}
Therefore{\samepage
\[
\bigtriangledown_r(\boxdot(V_2YU_2)) =
\boxdot(\bigtriangledown_r(V_2YU_2)) -\bigtriangledown_l(V_2YU_2)[Y,F].
\]
Since $w(V_2YU_2) < w(V_1YU_1)$ we can apply induction.}

Assume now that either $\mu = V\boxdot(V_1YU_1)$ or $\mu =\boxdot(V_1YU_1)U$. If $\mu = V\boxdot(V_1YU_1)$ then $\bigtriangledown_r(V\boxdot(V_1YU_1)) = V\bigtriangledown_r(\boxdot(V_1YU_1))$. If $\mu = \boxdot(V_1YU_1)U$ then $\bigtriangledown_r(\mu) = \bigtriangledown_{r,U}( \boxdot(V_1YU_1))$. Now,
\[
[\bigtriangledown_{r,U},\boxdot](V_1YU_1) =
\boxdot[\bigtriangledown_r(V_1YU_1)U - \bigtriangledown_{r,U}(V_1YU_1)] -
\bigtriangledown_l(V_1YU_1)\boxdot(U)
\]
and induction can be applied in these cases as well.

The last case is when $Y$ does not belong to a bracketed subword. Then $\mu =V_1YU_1$ and $\bigtriangledown_r(\mu) = V_1\boxdot(U_1)$.

The proof is completed.
\end{proof}

\begin{Corollary}\label{description of algebra of constants}The algebra of constants $K\langle X,Y\rangle^{\Delta}$ coincides with the algebra $R_F$.
\end{Corollary}

\begin{proof}As we already mentioned $R_F\subseteq K\langle X,Y\rangle^{\Delta}$ and it is sufficient to show that if $\Delta(L)=0$ for $L\in K\langle X,Y\rangle$, then $L$ belongs to $R_F$. But this is a direct consequence of the case $n=1$ in Theorem~\ref{preparing of freedom}.
\end{proof}

Now we are able to establish one of the main properties of the algebra of constants $K\langle X,Y\rangle^{\Delta}$.

\begin{Theorem}\label{free algebra of constants}The algebra of constants $K\langle X,Y\rangle^{\Delta}$ is a free algebra.
\end{Theorem}

\begin{proof}By Corollary \ref{description of algebra of constants} we may work with the algebra $R_F$ instead with $K\langle X,Y\rangle^{\Delta}$. When $m = 0$ we have seen (in the proof of Theorem~\ref{preparing of freedom}) that $R_1$ is generated by $X,T_1, T_2, \ldots$. Since $\overline{T_i} = Y^iX$ these elements freely generate $R_1$. For $m > 0$ producing a generating set is more involved but the freeness can be deduced from a theorem of de W.~Jooste~\cite{Jo}. It follows from his theorem that the kernel of the derivation $\overline{\Delta}(X) = 0$, $\overline{\Delta}(Y) =X^m$ is a free algebra. For this derivation any $w$-homogeneous component (recall that $w(X) = 1$, $w(Y) = m$) of a constant is also a constant, hence there is a~homogeneous free generating set $F_1, F_2, \ldots$ of $R_{X^m}$. There is a~bijection $\pi$ between the elements of $R_{X^m}$ and $R_F$ obtained by replacing $X^m$ in each bracket of an element of $R_{X^m}$ by $F=f(X)$. Therefore $\pi(F_1), \pi(F_2), \ldots$ is a generating set of $R_F$ which is free since $w(\pi(F_i) - F_i) < w(F_i)$.
\end{proof}

It remains to produce a homogeneous set freely generating $R_{X^m}$.

\begin{Lemma}\label{homogeneous free generating set} The algebra $R_{X^m}$ is generated by $X$ and bracketed words
\[
T_1^{i_1}X^{j_1} \cdots X^{j_{k-1}}T_1^{i_k},
\]
where $i_1, i_2, \ldots, i_k > 0$, $j_1, j_2, \ldots, j_{k-1} < m$, and where the right brackets $\}$ are preceded by~$T_1$ $($i.e., there are no configurations $X\})$.
\end{Lemma}

\begin{proof}Denote by $D$ the subalgebra of $R_{X^m}$ which is generated by words described in the lemma. Any element of $R_{X^m}$ can be written as a linear combination of bracketed words $\mu = X^{j_0}T_1^{i_1}X^{j_1} \cdots T_1^{i_k}X^{j_k}$. We shall find an element $B \in D$ with the same leading monomial $\overline{B}$ as the leading monomial $\overline{\mu}$ of $\mu$ in the lexicographic order defined by $Y\gg X > 1$. Clearly this is sufficient for the proof of the lemma.

To find the leading monomial $\overline{\mu}$ of a bracketed word $\mu$ we should replace all left brackets $\{$ by $Y$ and all right brackets $\}$ by $X^m$.

If $\overline{\mu}$ starts with $X$ then $\mu = X\mu_1$ (as an element of $K\langle X,Y\rangle$) where $\mu_1 \in R_{X^m}$ and we can use induction on weight to claim that there is an element $B_1 \in D$ such that $\overline{\mu_1} = \overline{B_1}$ (or even that $\mu_1 \in B$).

If $\mu$ cannot be written as $\boxdot(\nu)$ then $\mu = (\mu_1)(\mu_2)$ where brackets $()$ separate elements of $R_{X^m}$ and $w(\mu_i) < w(\mu)$. Hence we can use induction to claim that $\overline{\mu_1} = \overline{B_1}$, $\overline{\mu_2} = \overline{B_2}$ where $B_i \in D$.

If $\mu = \boxdot(\nu)$ then $w(\mu) = w(\nu) + 2m$ and we may assume that $\overline{\nu} = \overline{B}$ where $B \in D$. Since $B \in D$ we can write $B = \big(X^{j_0}\big)(V_1)\big(X^{j_1}\big)\cdots (V_k)\big(X^{j_k}\big)$ where $V_i \in D$ and $(X^j) = X^j$ and $\overline{\mu} = YX^{j_0}\overline{(V_1)\big(X^{j_1}\big)\cdots (V_k)}X^{j_k + m}$. Inasmuch as $V_i \in D$ we may assume that the first and the last letters in all $V_i$ (as bracketed words) are $T_1$.

If $j_0 > 0$ then $\overline{T_1}\big(X^{j_0-1}\big)\overline{(V_1)\big(X^{j_1}\big)\cdots (V_k)}\big(X^{j_k + m}\big) = \overline{\mu}$.

If $j_0 = 0$, $j_s \geq m$ where $s$ is the smallest possible then
\[
\overline{\big\{(V_1)\big(X^{j_1}\big)\cdots (V_s)\big\}\big(X^{j_s - m}\big)\cdots (V_k)}\big(X^{j_k + m}\big) = \overline{\mu}.
\]

If all $j_s < m$ then $\mu \in D$.
\end{proof}

\begin{Theorem}The algebra $D = R_{X^m}$, $m > 0$, is freely generated by $X$, $T_1$ and words $\boxdot\big(T_1^{i_1}X^{j_1} \cdots \allowbreak X^{j_{k-1}}T_1^{i_k}\big)$, where $i_1, i_2, \ldots,i_k > 0$, $j_1, j_2, \ldots, j_{k-1} < m$, and $T_1^{i_1}X^{j_1} \cdots X^{j_{k-1}}T_1^{i_k}$ are bracketed words described in Lemma~{\rm \ref{homogeneous free generating set}} $($we shall refer to these words as permissible and to $T_1^{i_1}X^{j_1} \cdots X^{j_{k-1}}T_1^{i_k}$ without brackets as the root of the corresponding word$)$.
\end{Theorem}

\begin{proof}It is sufficient to check that the leading monomial $\overline{\mu}$ of a permissible word cannot be presented as a product of the leading monomials of permissible words of a smaller weight.

To check this consider the leading monomial $\overline{\mu} = Y^{b_1} \cdots X^{a_{s-1}}Y^{b_s}X^{a_s}$ of a permissible $\mu$. (Observe that $b_1 > 0$, $a_s = m+1$ since $\overline{\boxdot(V)} =Y\overline{V}X^m$.)

The number of $T_1$ in the bracketed representation of $\mu \in D$ must be equal to $s$ since in the leading monomial of any word from $D$ a subword $YX$ can appear only as $\overline{T_1}$. So the number of brackets $\{$ in $\mu$ is $\deg_Y(\overline{\mu}) - s$. Of course the number of brackets $\}$ is the same.

A subword $Y^{b_i}X^{a_i}$ can appear in $\overline{\mu}$ only as $\{\ldots\{T_1\}\ldots\}X^{d_i}$ where the number of left brackets is $b_i - 1$, the number of right brackets is the integral part of $\frac{a_i - 1}{m}$ and $0 \leq d_i <m$ is the remainder of the division of $a_i - 1$ by $m$. Therefore the root and the bracketing of $\mu$ are uniquely determined by $\overline{\mu}$. But we would have two different bracketings if $\overline{\mu} = (\overline{\nu_1})(\overline{\nu_2})$. This finishes a proof of the theorem.
\end{proof}

\subsection*{Acknowledgements}

The authors are grateful to the Beijing International Center for Mathematical Research for warm hospitality during their visit when this work was started. The second author is grateful to the Max-Planck-Institut f\"ur Mathematik in Bonn where he was a visitor when this project was finished. While working on this project he was also supported by a FAPESP grant awarded by the State of Sao Paulo, Brazil.

\pdfbookmark[1]{References}{ref}
\LastPageEnding


\begin{thebibliography}{99}
\footnotesize\itemsep=0pt

\bibitem{B}
Bass H., A nontriangular action of {${\mathbb G}_{a}$} on {${\mathbb A}^{3}$},
 \href{https://doi.org/10.1016/0022-4049(84)90019-7}{\textit{J.~Pure Appl. Algebra}} \textbf{33} (1984), 1--5.

\bibitem{Co}
Cohn P.M., Free rings and their relations, 2nd ed., \textit{London Mathematical Society
 Monographs}, Vol.~19, Academic Press, Inc., London, 1985.

\bibitem{Cz1}
Czerniakiewicz A.J., Automorphisms of a free associative algebra of
 rank~{$2$}.~{I}, \href{https://doi.org/10.2307/1995814}{\textit{Trans. Amer. Math. Soc.}} \textbf{160} (1971),
 393--401.

\bibitem{Cz2}
Czerniakiewicz A.J., Automorphisms of a free associative algebra of
 rank~{$2$}.~{II}, \href{https://doi.org/10.2307/1996385}{\textit{Trans. Amer. Math. Soc.}} \textbf{171} (1972),
 309--315.

\bibitem{Jo}
de~W.~Jooste T., Primitive derivations in free associative algebras,
 \href{https://doi.org/10.1007/BF01214786}{\textit{Math.~Z.}} \textbf{164} (1978), 15--23.

\bibitem{D1}
Drensky V., Free algebras and {PI}-algebras. {G}raduate course in algebra,
 Springer-Verlag, Singapore, 2000.

\bibitem{DG}
Drensky V., Gupta C.K., Constants of {W}eitzenb\"{o}ck derivations and
 invariants of unipotent transformations acting on relatively free algebras,
 \href{https://doi.org/10.1016/j.jalgebra.2005.07.004}{\textit{J.~Algebra}} \textbf{292} (2005), 393--428, \href{https://arxiv.org/abs/math.RA/0412399}{arXiv:math.RA/0412399}.

\bibitem{Fa}
Falk G., Konstanzelemente in {R}ingen mit {D}ifferentiation, \href{https://doi.org/10.1007/BF01343560}{\textit{Math.
 Ann.}} \textbf{124} (1952), 182--186.

\bibitem{Fr1}
Freudenburg G., A survey of counterexamples to {H}ilbert's fourteenth problem,
 \textit{Serdica Math.~J.} \textbf{27} (2001), 171--192.

\bibitem{Fr2}
Freudenburg G., Algebraic theory of locally nilpotent derivations,
 \textit{Encyclopaedia of Mathematical Sciences}, Vol.~136, \href{https://doi.org/10.1007/978-3-540-29523-5}{Springer-Verlag},
 Berlin, 2006.

\bibitem{J}
Jung H.W.E., \"{U}ber ganze birationale {T}ransformationen der {E}bene,
 \href{https://doi.org/10.1515/crll.1942.184.161}{\textit{J.~Reine Angew. Math.}} \textbf{184} (1942), 161--174.

\bibitem{Kh}
Kharchenko V.K., Algebras of invariants of free algebras, \href{https://doi.org/10.1007/BF01674783}{\textit{Algebra
 Logic}} \textbf{17} (1978), 316--321, translation of \textit{Algebra i Logika} \textbf{17} (1978), 478--487.

\bibitem{L}
Lane D.R., Free algebras of rank two and their automorphisms, Ph.D.~Thesis,
 {B}edford College, London, 1976.

\bibitem{ML}
Makar-Limanov L., Automorphisms of a free algebra with two generators,
 \href{https://doi.org/10.1007/BF01075252}{\textit{Funct. Anal. Appl.}} \textbf{4} (1970), 262--264,
 translation of \textit{Funk. Analiz i ego Prilozh.} \textbf{4} (1970),  no.~3, 107--108.

\bibitem{ML2}
Makar-Limanov L., Locally nilpotent derivations, a new ring invariant and applications,
{L}ecture notes, available at \url{http://www.math.wayne.edu/~lml/lmlnotes.pdf}.

\bibitem{MLTU}
Makar-Limanov L., Turusbekova U., Umirbaev U., Automorphisms and derivations of
 free {P}oisson algebras in two variables, \href{https://doi.org/10.1016/j.jalgebra.2008.01.005}{\textit{J.~Algebra}} \textbf{322}
 (2009), 3318--3330, \href{https://arxiv.org/abs/0708.1148}{arXiv:0708.1148}.

\bibitem{Na}
Nagata M., On automorphism group of {$k[x,y]$}, \textit{Lectures in Math.},
 Vol.~5, Kyoto University, Kinokuniya Book-Store Co., Ltd., Tokyo, 1972.

\bibitem{No1}
Nowicki A., Polynomial derivations and their rings of constants, Uniwersytet
 Miko{\l}aja Kopernika, Toru\'{n}, 1994.

\bibitem{No2}
Nowicki A., The fourteenth problem of {H}ilbert for polynomial derivations, in
 Differential {G}alois {T}heory ({B}\c{e}dlewo, 2001), \textit{Banach Center
 Publ.}, Vol.~58, \href{https://doi.org/10.4064/bc58-0-13}{Polish Acad. Sci. Inst. Math.}, Warsaw, 2002, 177--188.

\bibitem{R}
Rentschler R., Op\'{e}rations du groupe additif sur le plan affine,
 \textit{C.~R.~Acad. Sci. Paris S\'{e}r.~A-B} \textbf{267} (1968), 384--387.

\bibitem{Se}
Serre J.-P., Trees, \textit{Springer Monographs in Mathematics}, Springer-Verlag, Berlin,
 2003, translation of Arbres, amalgames, ${\rm SL}_2$,
\textit{Ast\'erisque}, Vol.~46,
Soci\'et\'e Math\'ematique de France, Paris, 1977.

\bibitem{Ses}
Seshadri C.S., On a theorem of {W}eitzenb\"{o}ck in invariant theory,
 \href{https://doi.org/10.1215/kjm/1250525012}{\textit{J.~Math. Kyoto Univ.}} \textbf{1} (1962), 403--409.

\bibitem{SU1}
Shestakov I.P., Umirbaev U.U., Poisson brackets and two-generated subalgebras
 of rings of polynomials, \href{https://doi.org/10.1090/S0894-0347-03-00438-7}{\textit{J.~Amer. Math. Soc.}} \textbf{17} (2004),
 181--196.

\bibitem{SU2}
Shestakov I.P., Umirbaev U.U., The tame and the wild automorphisms of
 polynomial rings in three variables, \href{https://doi.org/10.1090/S0894-0347-03-00440-5}{\textit{J.~Amer. Math. Soc.}} \textbf{17}
 (2004), 197--227.

\bibitem{Sm}
Smith M.K., Stably tame automorphisms, \href{https://doi.org/10.1016/0022-4049(89)90158-8}{\textit{J.~Pure Appl. Algebra}}
 \textbf{58} (1989), 209--212.

\bibitem{Sp}
Specht W., Gesetze in {R}ingen.~{I}, \href{https://doi.org/10.1007/BF02230710}{\textit{Math.~Z.}} \textbf{52} (1950),
 557--589.

\bibitem{T}
Tyc A., An elementary proof of the {W}eitzenb\"{o}ck theorem, \href{https://doi.org/10.4064/cm-78-1-123-132}{\textit{Colloq.
 Math.}} \textbf{78} (1998), 123--132.

\bibitem{U}
Umirbaev U.U., The {A}nick automorphism of free associative algebras,
 \href{https://doi.org/10.1515/CRELLE.2007.030}{\textit{J.~Reine Angew. Math.}} \textbf{605} (2007), 165--178,
 \href{https://arxiv.org/abs/math.RA/0607029}{arXiv:math.RA/0607029}.

\bibitem{E}
van~den Essen A., Polynomial automorphisms and the {J}acobian conjecture,
 \textit{Progress in Mathematics}, Vol.~190, \href{https://doi.org/10.1007/978-3-0348-8440-2}{Birkh\"{a}user Verlag}, Basel,
 2000.

\bibitem{K}
van~der Kulk W., On polynomial rings in two variables, \textit{Nieuw Arch.
 Wiskunde} \textbf{1} (1953), 33--41.

\bibitem{W}
Weitzenb\"{o}ck R., \"{U}ber die {I}nvarianten von linearen {G}ruppen,
 \href{https://doi.org/10.1007/BF02547779}{\textit{Acta Math.}} \textbf{58} (1932), 231--293.

\end{thebibliography}
\end{document}